\documentclass[11pt
]{amsart}

\usepackage{hyperref}
\hypersetup{bookmarksdepth=3}

\usepackage{geometry}
\usepackage{amsmath,amssymb}
\usepackage{
mathrsfs}
\usepackage{esint}

\usepackage{tensor}

\usepackage{amssymb,amsmath,amsthm}

\newtheorem{theorem}{Theorem}
\newtheorem{lemma}[theorem]{Lemma}

\theoremstyle{remark}\newtheorem{remark}[theorem]{Remark}

\usepackage{graphicx}

\begin{document}

\title{The Muckenhoupt $A_{\infty}$ class as a metric space and continuity of weighted estimates}
\author{Nikolaos Pattakos and Alexander Volberg} \address{Department of Mathematics, Michigan State University, East
Lansing, MI 48824, USA}

\date{}

\begin{abstract}
{We show how the $A_{\infty}$ class of weights can be considered as a metric space. As far as we know this is the first time that a metric $d_{*}$ is considered on this set. We use this metric to generalize the results obtained in \cite{NV}. Namely, we show that for any Calder\'on-Zygmund operator $T$ and an $A_{p}$, $1<p<\infty$, weight $w_{0}$, the numbers $\|T\|_{L^{p}(w)\rightarrow L^{p}(w)}$ converge to $\|T\|_{L^{p}(w_{0})\rightarrow L^{p}(w_{0})}$ as $d_{*}(w,w_{0})\to 0$. We also find the rate of this convergence and prove that is sharp.}
\end{abstract}

\maketitle

\begin{section}{introduction and useful results}

The main purpose of this paper, is to define a natural metric structure on the classical Muckenhoupt $A_{p}$ classes, and generalize a continuity result obtained in \cite{NV}. Such continuity results have been coming up recently in connection with PDE with random coefficients and continuity of norms of Calder\'on-Zygmund operators. For example, the continuity at $w=1$, was used in \cite{CS}. 

The metric $A_{p}$ classes will be considered in section 2, where we study many properties of these new spaces, and the main theorem \ref{main}, in section 3. Before we state and prove the main theorems in sections 2 and 3, we need some definitions and some already known results about the weighted theory and it's relation with the $BMO$ space. 

We are going to work with functions $w\in L^{1}_{loc}(\mathbf R^{n})$ that are positive almost everywhere. Functions like these are known as weights. The celebrated $A_{p}$ classes of weights are defined in the following way:

For $1<p<\infty$, we say that $w\in A_{p}$ if for all cubes $Q$ in $\mathbf R^{n}$ we have that ($[w]_{A_{p}}$ is called the $A_{p}$ characteristic of the weight):

$$[w]_{A_{p}}:=\sup_{Q}\Big(\frac1{|Q|}\int_{Q}w\Big)\Big(\frac1{|Q|}\int_{Q}w^{1-p^{'}}\Big)^{p-1}<\infty,$$
where $p^{'}$ is the conjugate exponent to $p$, i.e. $\frac1{p}+\frac1{p^{'}}=1$. 

The class of $A_{1}$ weights consists of those $w$ such that there is a positive constant $c$ with the property:

$$Mw(x)\leq cw(x)$$
for almost every $x$ in $\mathbf R^{n}$, where $M$ is the Hardy-Littlewood maximal function. The smallest such constant is denoted by $[w]_{A_{1}}$ and is called the $A_{1}$ characteristic.

We define the class of $A_{\infty}$ weights as:

$$[w]_{A_{\infty}}:=\sup_{Q}\Big(\frac{\frac1{|Q|}\int_{Q}w}{\exp(\frac1{|Q|}\int_{Q}\log w)}\Big)<\infty.$$

It is really easy to see that any $A_{p}$ weight is actually an $A_{\infty}$ weight, and that we have the estimate $[w]_{A_{\infty}}\leq [w]_{A_{p}}$. It is also true that any $A_{\infty}$ weight is an $A_{p}$ weight for some $1<p<\infty$. This means that we have the equality:

$$A_{\infty}=\bigcup_{1<p<\infty} A_{p}.$$
Another nice property is that for $1\leq p\leq q\leq\infty$ we have $A_{1}\subset A_{p}\subset A_{q}\subset A_{\infty}$, where the inclusions here are strict. All of these sets are different for different values of $p$ and $q$. 

The space of $BMO$ functions in $\mathcal R^n$, consists of locally integrable functions $f$ such that the norm
$$\|f\|_{*}=\sup_{Q}\frac1{|Q|}\int_{Q}|f-f_{Q}|dx$$
if finite. The $BMO$ space and the $A_{\infty}$ space, have many nice properties. First of all, if $f$ is a $BMO$ function then for any number $\lambda\in(0,\frac{c}{\|f\|_{*}}]$, the function $e^{\lambda f}$ is an $A_{p}$ weight, $1<p<\infty$, where the constant $c$ depends on $p$ and the dimension $n$. Secondly, for small $BMO$ norm, the $A_{p}$ norm of the weight $e^{\lambda f}$ is bounded by the number $2$ for example (see e.g.  \cite{GCRF}).

A subset of $BMO$ that appears in many applications is $BLO$. It stands for the functions of bounded lower oscillation. A function $f\in L^{1}_{loc}(\mathbf R^{n})$ is said to belong in $BLO$ if there is a positive constant $c$ such that:

$$\frac{1}{|Q|}\int_{Q}f-\inf_{x\in Q}f(x)\leq c$$
for all cubes $Q$, where the infimum is understand as the essential infimum. It can be proved that for any $w\in A_{1}$, the function $\log w$ is in $BLO$. Also if a function $f\in BLO$ then for sufficiently small $\lambda>0$ the function $e^{\lambda f}\in A_{1}$. The reference for all these results is \cite{GCRF}.

For the proofs of our theorems, interpolation is going to play a really important role and for this reason we need some preliminary results on this subject as well. In the following $(X, \mathcal M, \mu)$ and $(Y,\mathcal N, \nu)$ will denote measure spaces. Suppose $T$ is an operator of a class of functions on $X$ into a class of functions on $Y$. $T$ is called a sub-linear operator, if it satisfies the following properties:\newline
i)If $f=f_{1}+f_{2}$ and $Tf_{1}, Tf_{2}$ are defined then $Tf$ is defined,\newline
ii)$|T(f_{1}+f_{2})|\leq|Tf_{1}|+|Tf_{2}|$, $\mu$ almost everywhere,\newline
iii)For any scalar $k$, we have $|T(kf)|=|k||Tf|$, $\mu$ almost everywhere.\newline
Let $p,q\geq1$ be two real numbers. We say that $T$ is of type $(p,q)$, if $T$ is defined for all functions $f$ in $L^{p}(X,\mathcal M,\mu)$ and there exists a positive number, $K$, independent of $f$, such that
$$\|Tf\|_{q,\nu}\leq K\|f\|_{p,\mu}$$ 
where
$$\|Tf\|_{q,\nu}=\Big(\int_{Y}|Tf|^{q}d\nu\Big)^{\frac{1}{q}}$$
and
$$\|f\|_{p,\mu}=\Big(\int_{X}|f|^{p}d\mu\Big)^{\frac1{p}}.$$
Let $\mu_{0},\mu_{1}$ be two measures for $(X,\mathcal M)$. If we define the measure $\mu=\mu_{0}+\mu_{1}$, then $\mu_{0},\mu_{1}$ are each absolutely continuous with respect to $\mu$. Thus, by the Radon-Nikodym theorem, there exists two functions, $\alpha_{0},\alpha_{1}$ such that for any $E\in\mathcal M$,

$$\mu_{j}(E)=\int_{E}\alpha_{j}d\mu$$
where $j=0,1$. In the following we will assume that $\alpha_{0},\alpha_{1}$ are never zero. This is equivalent to asserting that the sets of measure zero with respect to $\mu_{j}$, $j=0,1$, are the same as the sets of measure zero with respect to $\mu$. Thus, in the various measure spaces that we will consider, the equivalence classes of functions will be the same.

Let $0\leq s\leq1$, and define the measure $\mu_{s}$ on $X$ by

$$\mu_{s}(E)=\int_{E}\alpha_{0}^{1-s}\alpha_{1}^{s}d\mu,$$
for each $E\in\mathcal M$. Also assume, that we have two measures $\nu_{0},\nu_{1}$ on $\mathcal N$, and define the measures $\nu_{r}$, for $0\leq r\leq1$, just as we did for $\mu_{s}$ above.\newline
Given any real numbers $1\leq p_{0},p_{1}, q_{0}, q_{1}$ and any $0\leq t\leq1$, we define $p_{t}, q_{t}, s(t), r(t)$ as follows:
$$\frac{(1-t)p_{t}}{p_{0}}+\frac{tp_{t}}{p_{1}}=1, \frac{(1-t)q_{t}}{q_{0}}+\frac{tq_{t}}{q_{1}}=1$$
$$s(t)=\frac{(tp_{t})}{p_{1}}, r(t)=\frac{(tq_{t})}{q_{1}}.$$
We have the following theorem by \cite{SW}:

\begin{theorem}
\label{th1}
Suppose that $T$ is a sub-linear operator satisfying
$$\|Tf\|_{q_{j},\nu_{j}}\leq K_{j}\|f\|_{p_{j},\mu_{j}}$$
for all $f\in L^{p_{j}}(X,\mathcal M,\mu_{j})$, $j=0,1$. Then, for $0\leq t\leq1$, we have
$$\|Tf\|_{q_{t},\nu_{r(t)}}\leq K_{0}^{1-t}K_{1}^{t}\|f\|_{p_{t},\mu_{s(t)}}$$
for all $f\in L^{p_{t}}(X,\mathcal M, \mu_{s(t)})$.
\end{theorem}

In addition to the previous theorem we need also the following proved in \cite{MBKn}:
\begin{theorem}
\label{th2}
If the $A_{\infty}$ norm of a weight $w$ is small, i.e. $[w]_{A_{\infty}}\leq1+\delta<2,$ then the function $f=\log w$, and any cube $Q$ satisfy
$$\frac1{|Q|}\int_{Q}|f-f_{Q}|dx\leq 32\sqrt{\delta}.$$
\end{theorem}

Our purpose is to generalize the following theorem proved  in \cite{NV}:

\begin{theorem}
\label{th3}
Let $T$ be a linear operator such that for some $1<p<\infty$, 

$$\|T\|_{L^{p}(w)\rightarrow L^{p}(w)}\leq cF([w]_{A_{p}}),$$
for any $A_{p}$ weight w in $\mathbf R^{n}$, where $F$ is an increasing function and $c$ is an absolute constant. Then:

$$\lim_{[w]_{A_{p}}\to1}\|T\|_{L^{p}(w)\rightarrow L^{p}(w)}=\|T\|_{L^{p}\rightarrow L^{p}}.$$
\end{theorem}

It follows from the proof that if $[w]_{A_{p}}\leq1+\delta<2$, then:

$$\|T\|_{L^{p}(w)\rightarrow L^{p}(w)}\leq\|T\|_{L^{p}\rightarrow L^{p}}(1+c\sqrt{\delta}),$$
where $c$ is a constant that depends on the dimension $n$, on the constant $c$ that appears in the original weighted estimate and on $p$.

In order to do that we will define a metric in the $A_{\infty}$ space and this is going to generalize the convergence $[w]_{A_{p}}\to 1$ in the sense that this will be equivalent to the convergence $d_{*}(w,1)\to 0$ in the metric $d_{*}$.

\end{section}

\begin{section}{The $(\mathcal A_{\infty}, d_{*})$ metric space}

Let us observe that if we have any weight $w$, any positive constant $c>0$ and any $1\leq p\leq\infty$, then $[w]_{A_{p}}=[cw]_{A_{p}}$. We define an equivalence relation in $A_{\infty}$ in the following way: for $u, v\in A_{\infty}$ we will write $u\sim v$ if 
and only if there is a positive constant $c$ such that $u=cv$ almost everywhere in $\mathbf R^{n}$. It is trivial to check that this is an equivalence relation and this allows us to define the quotient space:

$$\mathcal A_{\infty}=A_{\infty}\Big/\sim.$$
In the same way we define for $1\leq p<\infty$:

$$\mathcal A_{p}=A_{p}\Big/\sim.$$
For two elements $u, v\in\mathcal A_{\infty}$ we define the distance function $d_{*}$ as:

$$d_{*}(u,v)=\|\log u-\log v\|_{*}.$$
Again it is obvious that all the requirements of a metric are satisfied and the reason for defining the equivalence relation is exactly because we need to have:

$$d_{*}(u,v)=0\Leftrightarrow u\sim v.$$
So we define a metric in $\mathcal A_{\infty}$, going through the $BMO$ space. We can check that for an $A_{p}$ weight $w$, $[w]_{A_{p}}\to 1$ is equivalent to $d_{*}(w,1)\to 0$ and since $\mathcal A_{p}\subset\mathcal A_{\infty},$ the restriction of the $d_{*}$ metric to $\mathcal A_{p}$, makes the class a metric space. 

The following is an obvious remark that gives more informations about this ``new" spaces. It states that small balls around the constant weight $1$, are complete in the $d_{*}$ metric. 

\begin{theorem}
\label{th4}
Consider a closed ball $\bar{B}(1,r)$ of sufficiently small radius $r>0$ and center the weight $1$, in the metric space $(\mathcal A_{\infty}, d_{*})$, i.e. $\bar{B}(1,r)=\{w\in\mathcal A_{\infty}:d_{*}(w,1)\leq r\}$.Then $\bar{B}(1,r)$ is a complete metric space with respect to the metric $d_{*}$.
\end{theorem} 
\begin{proof}:

Consider a Cauchy sequence $\{w_{n}\}_{n\in\mathbf N}$ in $(\bar{B}(1,r), d_{*})$. This means that the sequence $\{\log w_{n}\}_{n\in\mathbf N}$ is Cauchy in the $BMO$ space. But $BMO$ is Banach and so there is a function $f\in BMO$ such that $\log w_{n}\to f$ in $BMO$ as $n\to\infty$. By the John-Nirenberg inequality we know that there is a dimensional constant $c>0$ such that for all $\lambda\in(0,\frac{c}{\|f\|_{*}}]$ the function $e^{\lambda f}\in A_{2}$. But $|\|\log w_{n}\|_{*}-\|f\|_{*}|\leq\|\log w_{n}-f\|_{*}\to 0$ as $n\to\infty$. Here we use the fact that $w_{n}\in\bar{B}(1,r)$. This means that $\|\log w_{n}-\log 1\|_{*}=\|\log w_{n}\|_{*}\leq r$  and $r$ is sufficiently small. Therefore, the number $\|f\|_{*}$ is small and so the number $\frac{c}{\|f\|_{*}}$ is really big. We are now allowed to choose for $\lambda=1$ and we get that $e^{f}\in A_{2}$ or equivalently there is a weight $w\in A_{2}\subset A_{\infty}$ with $f=\log w$. It is trivial now to see that $d_{*}(w_{n},w)\to 0$ as $n\to \infty$. 

\end{proof}

Of course in the previous theorem, we can replace the $\mathcal A_{\infty}$ space by any of the other $\mathcal A_{p}$ spaces. We should mention that no one of the $\mathcal A_{p}$ spaces is complete. The proof of this fact is very simple. Let us prove that $\mathcal A_{1}$ is not complete by finding a Cauchy sequence in the space that has no limit inside $\mathcal A_{1}$. It will follow that this example works for anyone of the $\mathcal A_{p}$ spaces. Consider a decreasing sequence $-1<r_{n}<0$ with $\lim_{n\to\infty}r_{n}=-1$. Define the $A_{1}$ weights $w_{n}=|x|^{r_{n}}$. Then:

$$d_{*}(w_{r_{n}},w_{r_{m}})=\|r_{n}\log|x|-r_{m}\log|x|\|_{*}=|r_{n}-r_{m}|\|\log|x|\|_{*}$$
and since $r_{n}\to 1$ we see that $\{w_{n}\}_{n\in\mathbf N}$ is Cauchy in $\mathcal A_{1}$, or equivalently the sequence $\{\log w_{n}\}_{n\in\mathbf N}$ is Cauchy in $BMO$. It's limit in the $BMO$ space is obviously the function $f(x)=-\log|x|$. This means that for $w(x)=\frac{1}{|x|}$ we have $d_{*}(w_{n},w)\to 0$ as $n\to \infty$, but since $w$ is not in $L^{1}_{loc}(\mathbf R^{n})$ it can not be an $A_{1}$ weight. So the space $(\mathcal A_{1}, d_{*})$ is not complete. 

Let us also mention the following result in \cite{GJ}, by Garnett and Jones, that helps to understand better when a ball in $(\mathcal A_{p}, d_{*})$ is complete. It states that  for a function $f\in BMO$,

$$dist_{BMO}(f, L^{\infty}):=\inf\{\|f-g\|_{*}:g\in L^{\infty}\}\sim\frac1{\sup\{\lambda>0:e^{\lambda f}\in A_{2}\}}.$$
This means that if we have a Cauchy sequence in $\mathcal A_{p}$, the closer the sequence is to the $L^{\infty}$ space, the more chances it has to have a limit in $\mathcal A_{p}$. 

So now we can try and find the completion of these spaces under the metric $d_{*}$. By definition the completion of $(\mathcal A_{p} , d_{*})$ is the space $\bar{\mathcal A_{p}}$ that consists of the equivalence classes of all Cauchy sequences of $\mathcal A_{p}$. We can identify this space as a subspace of $BMO$. Indeed:

$$\bar{\mathcal A_{p}}=\{ f \in BMO : \exists \{w_{n}\}_{n\in\mathbf N}\subset A_{p}:\lim_{n\to\infty}\|\log w_{n}-f\|_{*}=0\},$$
and we can think of the $\mathcal A_{p}$ class as a subset of $\bar{\mathcal A_{p}}$, by identifying every weight $w$ with it's logarithm, $\log w$, in $BMO$. Since the classical $A_{p}$ spaces form an increasing ``sequence" of the variable $p$ (and of course the same is true for the $\mathcal A_{p}$ spaces), the same is true for this new subspaces of $BMO$, $\bar{\mathcal A_{1}}\subset \bar{\mathcal A_{p}}\subset \bar{\mathcal A_{q}}\subset \bar{\mathcal A_{\infty}}\subset BMO$, for $1\leq p\leq q\leq\infty$. 

They are also $\bf{convex}$ subsets of $BMO$. Indeed, consider $1<p<\infty$, and $f,g \in\bar{\mathcal A_{p}}$. This means that there are sequences $\{w_{n}\}_{n\in\mathbf N},\{v_{n}\}_{n\in\mathbf N}\subset A_{p}$ such that: $f=\lim_{n\to\infty}w_{n}, g=\lim_{n\to\infty}v_{n}$, in $BMO$. Let $0<t<1$ be fixed. We will show that $tf+(1-t)g\in\bar{\mathcal A_{p}}$. For this, we only need to see that $tf+(1-t)g=\lim_{n\to\infty}\log (w_{n}^{t}v_{n}^{1-t})$, in $BMO$, and check using H\"older that the weight $w_{n}^{t}v_{n}^{1-t}\in A_{p}$, for all $n$, since:

$$[w^{t}v^{1-t}]_{A_{p}}\leq[w]_{A_{p}}^{t}[v]_{A_{p}}^{1-t},$$
for all $w,v\in A_{p}$. Thus, $tf+(1-t)g\in\bar{\mathcal A_{p}}$. It is trivial to see now that $\bar{\mathcal A_{\infty}}$ is also a convex subset of $BMO$. For $\bar{\mathcal A_{1}}$ the same holds, since if we have two $A_{1}$ weights, $w,v$, it is trivial to see that $w^{t}v^{1-t}\in A_{1}$ and actually that $[w^{t}v^{1-t}]_{A_{1}}\leq [w]_{A_{1}}^{t}[v]_{A_{1}}^{1-t}$.    

Here, let us observe that for any $1<p<\infty$, we have that $L^{\infty}\subset\bar{\mathcal A_{p}}$. There is a nice result of weighted theory (see\cite{GCRF}) that states the following (we will present the statement only for $A_{2}$): There are dimensional constants $c_{1}, c_{2}>0$, such that for a function $\phi$ in $\mathbf R^{n}$ we have:\newline
a) $e^{\phi}\in A_{2}$ provided $\inf\{\|\phi-g\|_{*}:g\in L^{\infty}\}\leq c_{1}$ and\newline
b) $\inf\{\|\phi-g\|_{*}:g\in L^{\infty}\}\leq c_{2}$ provided $e^{\phi}\in A_{2}$. This means that all functions $f\in BMO$ that satisfy the assumption a, belong to the $\bar{\mathcal A_{2}}$ space. Equivalently, there is a small neighborhood of $L^{\infty}$ inside $BMO$, that lies inside the $\bar{\mathcal A_{2}}$ space. 

We should also mention that since:

$$BLO=\{\alpha\log w:\alpha\geq0, w\in A_{1}\},$$
we can ask the question if the spaces $\bar{\mathcal A_{1}}$, $BLO$ are equal. Let us assume that they are. A classical result of weighted theory is that $BMO=BLO-BLO$. By our assumption we have that $BMO=\bar{\mathcal A_{1}}-\bar{\mathcal A_{1}}$. Now consider a function $f\in BMO$. There are functions $\phi,\psi\in\bar{\mathcal A_{1}}$ such that $f=\phi-\psi$. We know that there are sequences of $A_{1}$ weights $\{\phi_{n}\}_{n\in\mathbf N}, \{\psi_{n}\}_{n\in\mathbf N}$ such that $f=\lim_{n\to\infty}\log \phi_{n}-\lim_{n\to\infty}\log \psi_{n}=\lim_{n\to\infty}\log\phi_{n}\psi_{n}^{-1}$, where the limit is in $BMO$. But $\phi_{n}\psi_{n}^{-1}$ is an $A_{2}$ weight for all $n$. So we get that $\bar{\mathcal A_{2}}=BMO$. But this is obviously false.

Notice that from the argument follows the inclusion, $\bar{\mathcal A_{1}}-\bar{\mathcal A_{1}}\subset\bar{\mathcal A_{2}}.$ Trivially, we have the more general fact, that for any $1<p<\infty$, $\bar{\mathcal A_{1}}+(1-p)\bar{\mathcal A_{1}}\subset \bar{\mathcal A_{p}}$. Also, since we have that $w\in A_{p}\Leftrightarrow w^{1-p^{'}}\in A_{p^{'}}$, we get the equivalence $f\in\bar{\mathcal A_{p}}\Leftrightarrow (1-p^{'})f\in\bar{\mathcal A_{p^{'}}}$. For $p=2$ we have $f\in\bar{\mathcal A_{2}}\Leftrightarrow -f\in\bar{\mathcal A_{2}}$, which means that the $\bar{\mathcal A_{2}}$ class is symmetric with respect to the origin in the $BMO$ space. No other $\bar{\mathcal A_{p}}$ class has this property. Here we should remember the following about power weights. A function of the form $|x|^{\alpha}$ is an $A_{p}$ weight in $\mathbf R^{n}$, if and only if $-n<\alpha<n(p-1)$. The interval for $\alpha$ is symmetric with respect to the origin, if and only if $p=2$. Now we can see that there is a  ``correspondence" between the $\bar{\mathcal A_{2}}$ space and the interval $(-n, n)$.

\end{section}

\begin{section}{A generalized version of theorem \ref{th3}}

Our goal is to prove the following theorem:

\begin{theorem}
\label{main}
Let $T$ be a linear operator such that for some $1<p<\infty$, 

$$\|T\|_{L^{p}(w)\rightarrow L^{p}(w)}\leq cF([w]_{A_{p}}),$$
for any $A_{p}$ weight w in $\mathbf R^{n}$, where $F$ is an increasing function and $c$ is an absolute constant. Fix an $A_{p}$ weight $w_{0}$. Then:

$$\lim_{d_{*}(w,w_{0})\to 0}\|T\|_{L^{p}(w)\rightarrow L^{p}(w)}=\|T\|_{L^{p}(w_{0})\rightarrow L^{p}(w_{0})},$$
and actually from the proof follows that for any sublinear operator satisfying the hypothesis of the theorem we have the estimate:

$$\|T\|_{L^{p}(w)\rightarrow L^{p}(w)}\leq\|T\|_{L^{p}(w_{0})\rightarrow L^{p}(w_{0})}(1+c\delta)$$ 
for all weights $w\in A_{p}$ with $d_{*}(w,w_{0})\leq\delta$, for sufficiently small $\delta$, where $c$ is a positive constant that depends on $p$, on the constant $c$ that appears in the original weighted estimate, on the dimension $n$ and on $[w_{0}]_{A_{p}}$.

\end{theorem}

Here let us mention something that is important. Say that our $A_{\infty}$ weight $w$ is of the ``order" $[w]_{A_{\infty}}<1+\delta$. Then by theorem \ref{th2} we get that $d_{*}(w,1)\leq c\sqrt{\delta}$. Since theorem \ref{main} is going to be a generalization of theorem \ref{th3}, the rate of convergence that we have in both theorems should agree. The above observation explains exactly this.

\begin{remark}
Notice that the second half of the previous theorem is true for the Maximal function (since it is true for all sublinear operators that are bounded in the way described in the theorem), i.e. for all weights $w\in A_{p}$ that are sufficiently close to $w_{0}\in A_{p}$, with $d_{*}(w,w_{0})\leq\delta$:

$$\|M\|_{L^{p}(w)\rightarrow L^{p}(w)}\leq\|M\|_{L^{p}(w_{0})\rightarrow L^{p}(w_{0})}(1+c\delta).$$
It is well-known (see  \cite{B}) that $\|M\|_{L^{p}(w)\rightarrow L^{p}(w)}\leq c[w]_{A_{p}}^{\frac{1}{p-1}}$, which can be used here.  
\end{remark}
 
The argument is really similar to the one given in \cite{NV}. Nevertheless, we are going to present it for the sake of completeness and because there are some small differences that we have to point out.

\begin{proof}:
It consists of two steps. First we will show that for any $\bf sublinear$ operator that satisfies the assumptions of our theorem we have:

$$\|T\|_{L^{p}(w)\rightarrow L^{p}(w)}\leq\|T\|_{L^{p}(w_{0})\rightarrow L^{p}(w_{0})}(1+c\delta)$$ 
for all weights $w\in A_{p}$ with $d_{*}(w,w_{0})\leq\delta$. Let $0<\delta$ be a really small number that we consider to be fixed. Fix also an $A_{p}$ weight $w$, with $d_{*}(w,w_{0})<\delta$. This means that $\|\log\frac{w}{w_{0}}\|_{*}\leq\delta$. We would like to write our weight $w$ as $w=w_{0}^{1-t}W^{t}$, for some small and positive number $t$ (which is going to be like $\delta$), and some weight $W\in A_{p}$. From this expression we can see that $W=\frac{w^{\frac1{t}}}{w_{0}^{\frac1{t}}}w_{0}$. For this, let us consider only the case $p=2$, but the general case is identical to this one. Since $w_{0}\in A_{2}$ we know that there is a small $\epsilon>0$ such that $w_{1}:=w_{0}^{1+\epsilon}\in A_{2}$. Then obviously $w_{0}=w_{1}^{1-s}$ for small $s>0$. To continue, consider the function $f=\log\Big(\frac{w}{w_{0}}\Big)^{\frac1{s}}$. The $BMO$ norm of $f$ is really small since:

$$\|f\|_{*}=\frac{1}{s}d_{*}(w,w_{0})\leq\frac1{s}\delta,$$
and so by the John-Nirenberg inequality we have that for all $\lambda\in(0,\frac{c}{\|f\|_{*}}]$ the function $e^{\lambda f}=\Big(\frac{w}{w_{0}}\Big)^{\frac{\lambda}{s}}\in A_{2}$, where $c$ is a positive constant that depends only on the dimension. If we choose $\lambda=\frac{c_{0}}{\delta}$, $c_{0}>0$ is any constant less than or equal to $sc$, we see that $w_{2}:=\Big(\frac{w}{w_{0}}\Big)^{\frac{c_{0}}{\delta s}}\in A_{2}$, which means that the function:

$$W=\frac{w^{\frac1{t}}}{w_{0}^{\frac1{t}}}w_{0}=w_{1}^{1-s}w_{2}^{s}\in A_{2},$$
where $\delta=c_{0}t$. Here we should mention that the $A_{2}$ norm of $W$ can be chosen to be bounded above by a constant that depends only on $s$. On the other hand, $s$ depends only on the $A_{2}$ norm of $w_{0}$, and this is fixed. With this in mind, let us assume that the $A_{2}$ characteristic of $W$ is bounded above by $4$. The important thing here is that it does not depend on $\delta$.

Now the proof continues like this: Write $\gamma=\|T\|_{L^{p}(w_{0})\rightarrow L^{p}(w_{0})}$. By the interpolation result of Stein and Weiss, theorem \ref{th1}, for $X=Y=\mathcal R^n$, $\mathcal M=\mathcal N=\mathcal L$ and $\mu_{0}=\nu_{0}=w_{0}dx, \mu_{1}=\nu_{1}=Wdx$, where by $\mathcal L$ we denote the $\sigma$-algebra of Lebesgue measurable sets in $\mathbf R^{n}$, we get:

\begin{eqnarray*}
\|T\|_{L^{p}(w)\rightarrow L^{p}(w)}&\leq&\gamma^{1-t}\|T\|_{L^{p}(W)\rightarrow L^{p}(W)}^{t}\\
&\leq&\gamma^{1-t}c^{t}F\Big([W]_{A_{p}}\Big)^{t}\\
&\leq&\gamma^{1-t}c^{t}F(4)^{t}
\end{eqnarray*}
and the right-hand side goes to $\gamma$ as $t\to0^{+}$ or equivalently as $\delta\to0^{+}$. In other words:
$$\limsup_{d_{*}(w,w_{0})\to 0}\|T\|_{L^{p}(w)\rightarrow L^{p}(w)}\leq\|T\|_{L^{p}(w_{0})\rightarrow L^{p}(w_{0})}$$
and in addition we have the desired estimate: 

$$\|T\|_{L^{p}(w)\rightarrow L^{p}(w)}\leq\|T\|_{L^{p}(w_{0})\rightarrow L^{p}(w_{0})}(1+c\delta),$$
where $c$ is a constant depending on $n, p$ and $[w_{0}]_{A_{p}}$, for all weights $w$ in $A_{p}$ that are $\delta$ close to $w_{0}$ in the $d_{*}$ metric.  

\begin{remark}
Notice that the previous calculations show that the set 

$$\{\log w:w\in A_{p}\},$$
is open in $BMO$ for all $1<p<\infty$. To see this fix $w_{0}\in A_{p}$ and choose sufficiently small $\delta>0$. For $f\in BMO$ with $\|f-\log w_{0}\|_{*}\leq\delta$, write $f=\log u$, where $u$ is a positive function. Then follow the previous reasoning in the beginning of the proof, with $w= u$ and write $u=w_{0}^{1-t}W^{t}$, for $0<t<1$. It follows that $W\in A_{p}$ and so $u=w_{0}^{1-t}W^{t}$ is an $A_{p}$ weight, by H\"older's inequality. As we can see, this is exaclty the same argument as before. 
\end{remark}

Now we show that for a $\bf linear$ operator we have the estimate:

$$\|T\|_{L^{p}(w_{0})\rightarrow L^{p}(w_{0})}\leq\liminf_{d_{*}(w,w_{0})\to 0}\|T\|_{L^{p}(w)\rightarrow L^{p}(w)}.$$

Let us assume for simplicity that $p=2$ and that $\|T\|_{L^{2}(w_{0})\rightarrow L^{2}(w_{0})}=1$. Note that other $p'$s can be treated similarly. So far we have proved that:

$$\limsup_{d_{*}(w,w_{0})\to 0}\|T\|_{L^{2}(w)\rightarrow L^{2}(w)}\leq 1$$
and:

$$d_{*}(w,w_{0})\leq\delta<1\Rightarrow\|T\|_{L^{2}(w)\rightarrow L^{2}(w)}\leq 1+c\delta.$$ 
Let $M_{\phi}$ denote the operation of multiplication by $\phi$. To finish the proof of the continuity at $w=w_{0}$ we are going to assume that:

$$\liminf_{d_{*}(w,w_{0})\to 0}\|T\|_{L^{2}(w)\rightarrow L^{2}(w)}=\liminf_{d_{*}(w,w_{0})\to 0}\Big\|M_{w_{0}^{-\frac12}w^{\frac12}}TM_{w_{0}^{\frac12}w^{-\frac12}}\Big\|_{L^{2}(w_{0})\rightarrow L^{2}(w_{0})}<1$$
and get a contradiction.  This means that there is $\tau>0$ small, and a sequence of $A_{2}$ weights $w_{n}$ such that $d_{*}(w_{n},w_{0})\to 0$ as $n\to\infty$ and in addition:

\begin{equation}
\label{ine1}
\|w_{0}^{-\frac12}w_{n}^{\frac12}Tw_{0}^{\frac12}w_{n}^{-\frac12}g \|_{L^{2}(w_{0})}\leq(1-\tau)\|g\|_{L^{2}(w_{0})}
\end{equation}
for all  functions $g\in L^2(w_0)$.

 Fix now any cube $Q$. Here we can make the normalization assumption $\frac1{|Q|}\int_{Q}\frac{w_{n}}{w_{0}}dx=1$ for all $n\in\mathbf N$. We claim two things:\newline
$1^{*}$)$\|w_{n}^{-\frac12}-w_{0}^{-\frac12}\|_{L^{2}(w_{0},Q)}\to 0$ as $n\to\infty$ where by $L^{2}(w_{0},Q)$ we mean the $L^{2}(w_{0})$ norm over $Q$,  and\newline
$2^{*}$) there exists a subsequence $k_{n}$ such that $w_{k_{n}}\to w_{0}$ almost everywhere in the cube $Q$.\newline
Obviously $2^{*}$ follows from $1^{*}$. For a proof of $1^{*}$, see lemma after the end of this proof. Now without loss of generality we can assume that the subsequence is the original sequence $w_{n}$. Notice now that $1^{*}$ implies $\|w_{n}^{-\frac12}f-w_{0}^{-\frac12}f\|_{L^{2}(w_{0},Q)}\to 0$ as $n\to\infty$ for all bounded $f$, and so for $g=fw_{0}^{-\frac12}$, we get $\|T(w_{0}^{\frac12}w_{n}^{-\frac12}g)-Tg\|_{L^{2}(w_{0},Q)}\to 0$ as $n\to\infty$ and this implies that for a subsequence of $w_{n}$ (which again we assume that is the whole sequence), $w_{0}^{-\frac12}w_{n}^{\frac12}Tw_{0}^{\frac12}w_{n}^{-\frac12}g\to Tg$ almost everywhere in the cube $Q$. Now we apply Fatou's lemma in inequality \eqref{ine1} and we get:

$$\Big\|\liminf_{n\to\infty}w_{0}^{-\frac12}w_{n}^{\frac12}Tw_{0}^{\frac12}w_{n}^{-\frac12}g\Big\|_{L^{2}(w_{0},Q)}\leq\liminf_{n\to\infty}\Big\|w_{0}^{-\frac12}w_{n}^{\frac12}Tw_{0}^{\frac12}w_{n}^{-\frac12}g\Big\|_{L^{2}(w_{0},Q)}\leq(1-\tau)\|g\|_{L^{2}(w_{0},Q)}.$$

Here $g=fw_{0}^{-\frac12}$ with bounded $f$ form a dense family in $L^2(w_0, Q)$.
For $g$ from this dense family it follows:

$$\|Tg\|_{L^{2}(w_{0})}\leq(1-\tau)\|g\|_{L^{2}(w_{0})}$$
by letting the cube $Q$ expand to infinity, for $g$ in some dense subclass of $L^{2}(w_{0})$ . By assumption $\|T\|_{L^{2}(w_{0})\rightarrow L^{2}(w_{0})}=1$ and this is how we have our contradiction. 

\end{proof}

All that remains is the following lemma:

\begin{lemma}

Let $w_{0}, w\in A_{2}$ such that $d_{*}(w,w_{0})\leq\epsilon$, where $\epsilon$ is sufficiently small. Let us have a normalization assumption $\frac1{|Q|}\int_{Q}\frac{w}{w_{0}}dx=1$. Then $\|w_{n}^{-\frac12}-w_{0}^{-\frac12}\|_{L^{2}(w_{0},Q)}\leq|Q|^{\frac12}c(\epsilon)^{\frac12}$, where $c(\epsilon)$ goes to $0$ as $\epsilon$ goes to $0$.  

\end{lemma} 
\begin{proof}:
We want to estimate the expression:

$$\frac{1}{|Q|}\Big\|w^{-\frac12}-w_{0}^{-\frac12}\Big\|_{L^{2}(w_{0},Q)}^{2}=\frac{1}{|Q|}\int_{Q}\frac{w_{0}}{w}+1-\frac{2}{|Q|}\int_{Q}\Big(\frac{w_{0}}{w}\Big)^{\frac12}.$$
The last integral can be taken care of really easy, since by our normalization assumption and Cauchy-Schwartz we get the following:

$$\frac1{|Q|}\int_{Q}\Big(\frac{w_{0}}{w}\Big)^{\frac12}=\frac{1}{|Q|}\int_{Q}\Big(\frac{w}{w_{0}}\Big)^{-\frac12}\geq\Big(\frac{1}{|Q|}\int_{Q}\Big(\frac{w}{w_{0}}\Big)^{\frac12}\Big)^{-1}\geq\Big(\frac1{|Q|}\int_{Q}\frac{w}{w_{0}}\Big)^{-\frac12}=1.$$
Therefore, the quantity that we need to estimate is bounded above by:

$$\frac{1}{|Q|}\Big\|w^{-\frac12}-w_{0}^{-\frac12}\Big\|_{L^{2}(w_{0},Q)}^{2}\leq\frac{1}{|Q|}\int_{Q}\frac{w_{0}}{w}-1.$$
It is time to use the fact that $d_{*}(w,w_{0})\leq\epsilon$. We get that the weight $\frac{w}{w_{0}}$ is in the $A_{2}$ class and actually because the $BMO$ norm of $\log\Big(\frac{w}{w_{0}}\Big)$ is really small, the $A_{2}$ characteristic is bounded by $1+c(\epsilon)$, where $c(\epsilon)$ is a constant that goes to $0$ as $\epsilon$ goes to $0$. So:

$$\frac{1}{|Q|}\Big\|w^{-\frac12}-w_{0}^{-\frac12}\Big\|_{L^{2}(w_{0},Q)}^{2}\leq\Big[\frac{w}{w_{0}}\Big]_{A_{2}}-1\leq c(\epsilon).$$
\end{proof}

\begin{section}{comments and observations}
Let us have a closer look to what the previous theorem tells us. Consider any linear operator $T$ that satisfies the assumptions of theorem \ref{main}. This means that for any $w\in A_{p}$ we have a number $\|T\|_{L^{p}(w)\rightarrow L^{p}(w)}$. So we have a map $F_{T}:\mathcal A_{p}\rightarrow \mathbf R$ defined by the formula:

$$F_{T}(w)=\|T\|_{L^{p}(w)\rightarrow L^{p}(w)}.$$
First we should observe that for $w\in\mathcal A_{p}$ this is well defined since $\|T\|_{L^{p}(w)\rightarrow L^{p}(w)}=\|T\|_{L^{p}(cw)\rightarrow L^{p}(cw)}$ for all $w\in A_{p}$ and all positive constants $c>0$. By theorem \ref{main} we have that this map is continuous, since:

$$\lim_{d_{*}(w,w_{0})\to 0}|F_{T}(w)-F_{T}(w_{0})|=0$$
for all weights $w,w_{0}\in\mathcal A_{p}$. 

In \cite{NV} the authors showed that for the Hilbert transform, $H$, in $\mathbf S^{1}$ we have that for all sufficiently small $\delta$'s, there is a weight $w\in A_{2}$, with the properties that $[w]_{A_{2}}\leq 1+\delta<2$ and:

$$c\sqrt{\delta}\leq\|H\|_{L^{2}(w)\rightarrow L^{2}(w)}-1.$$
This means that theorem \ref{th3} is sharp for $p=2$. This is true also for the Hilbert transform in the line and for the martingale transform. If we use the observation made exactly after the statement of theorem \ref{main}, we see that the rate of convergence in theorem \ref{main} is also sharp. It is a good point to mention that there are interesting operators like the Riesz projection $P_{+}$ in $\mathbf S^{1}$, that converge faster to their $L^{2}(dx)$ norm than the previous mentioned operators (see \cite{NV}). Namely, there is universal constant $c>0$ such that for all weights $[w]_{A_{2}}\leq 1+\delta<2$, we have:

$$\|P_{+}\|_{L^{2}(w)\rightarrow L^{2}(w)}-1\leq c\delta.$$
In addition, there is a universal constant $c_{1}>0$, such that for all sufficiently small $\delta$'s, there is weight $w\in A_{2}$ with the properties $[w]_{A_{2}}\leq1+\delta$ and:

$$c_{1}\delta\leq\|P_{+}\|_{L^{2}(w)\rightarrow L^{2}(w)}-1.$$
We should also mention that in the proof of theorem \ref{main}, there is only one time (namely in the second step) that we really need to use the fact that our operator is linear in order to get that:

$$\|T\|_{L^{p}(w_{0})\rightarrow L^{p}(w_{0})}\leq\liminf_{d_{*}(w,w_{0})\to 0}\|T\|_{L^{p}(w)\rightarrow L^{p}(w)}.$$
It is used when we claim that the convergence $\|w_{n}^{-\frac12}f-w_{0}^{-\frac12}f\|_{L^{2}(w_{0},Q)}\to 0$ as $n\to\infty$, implies the convergence $\|T(w_{0}^{\frac12}w_{n}^{-\frac12}f)-Tf\|_{L^{2}(w_{0},Q)}\to 0$ as $n\to\infty$.

By \cite{Hy}, \cite{TCSA}, we know that for any Calder\'on-Zygmund operator $T$, any $1<p<\infty$, and any $A_{p}$ weight w, we have the estimate:

$$\|T\|_{L^{p}(w)\rightarrow L^{p}(w)}\leq c[w]_{A_{p}}^{max\{1,\frac{1}{p-1}\}},$$
where $c$ is a universal constant that does not depend on the weight, and so we see that theorem \ref{main} can be applied for this class of operators, since the function $F$ that appears in the statement of this theorem can be chosen to be equal to $F(x)=x^{max\{1,\frac{1}{p-1}\}}$.
\end{section}

\end{section}

Alexander Volberg, volberg@math.msu.edu, Nikolaos Pattakos, pattakos@msu.edu;\newline
Department of Mathematics, Michigan State University, East Lansing, Michigan 48824, USA;


\begin{thebibliography}{00}

\bibitem{B} Stephen M. Buckley,
\emph{Estimates for operator norms on weighted spaces and reverse Jensen inequalities.}
Transactions of the AMS, volume 340, number 1, November 1993.

\bibitem{CS} J.G. Conlon and T. Spencer, 
\emph{A strong central limit theorem for a class of random surfaces.} 
 arXiv:1105.2814.

\bibitem{GCRF} J. Garcia-Cuerva and J. Rubio De Francia,
\emph{Weighted norm inequalities and related topics.}
North Holland Math. Stud. 116, North Holland, Amsterdam 1985.

\bibitem{GJ} J. Garnett and P.W. Jones,
\emph{The distance in $BMO$ to $L^{\infty}$.}
Ann. of Math.108(1978), 373-393.

\bibitem{Hy} T.  Hyt\"onen,
\emph{The sharp weighted bound for general Calder\'on-Zygmund operators.} 
arXiv:1007.4330.

\bibitem{TCSA} T. Hyt\"onen, C. P\'erez, S. Treil, and A. Volberg,
\emph{Sharp weighted estimates of the dyadic shifts and $A_2$ conjecture.} 
arXiv:1010.0755, 2010.

\bibitem{MBKn} Michael Brian Korey,
\emph{Ideal weights: Asymptotically optimal versions of doubling, absolute continuity, and bounded mean oscillation .}
The Journal of Fourier Analysis and Applications, Volume 4, Issues 4 and 5, 1998.

\bibitem{MBK} Michael Brian Korey,  
\emph{Correction to ``optimal factorization of weights".} 
Transactions of the AMS, Volume 353, Number 2, Pages 839-851 2000.


\bibitem{NV} N. Pattakos and A. Volberg, 
\emph{Continuity of weighted estimates in $A_{p}$ norm, Proceedings of the AMS, to appear.}


\bibitem{SW} E. Stein and G. Weiss,  
\emph{Interpolation of operators with change of measures. } 
Transactions of the AMS, Volume 87, Number 1, Pages 159-172 1958.



\end{thebibliography}
\end{document}